\documentclass{article}
\usepackage{amsthm,amsmath,amsfonts,amssymb}

\title{Eventual regularization of the slightly supercritical fractional
Burgers equation}
\author{Chi Hin Chan, Magdalena Czubak and Luis Silvestre}

\newlength{\hchng}
\newlength{\vchng}
\newtheorem{thm}{Theorem}[section]

\newtheorem{cor}[thm]{Corollary}
\newtheorem{lemma}[thm]{Lemma}

\newtheorem{preremark}[thm]{Remark}
\newenvironment{remark}{\begin{preremark}\rm}{\medskip \end{preremark}}
\numberwithin{equation}{section}

\newcommand{\norm}[1]{\left\Vert#1\right\Vert}
\newcommand{\abs}[1]{\left\vert#1\right\vert}

\newcommand{\R}{\mathbb R}
\newcommand{\eps}{\varepsilon}

\newcommand{\lap} {\Delta}

\newcommand{\dd} {\; \mathrm{d}}

\DeclareMathOperator*{\osc}{osc}

\def\be{\begin{equation}}
\def\ee{\end{equation}}

\begin{document}
\maketitle

\begin{abstract}
We prove that a weak solution of a slightly supercritical fractional Burgers
equation becomes H\"older continuous for large time.
\end{abstract}

\section{Introduction}
We consider the fractional Burgers equation
\begin{equation} \label{e:burgers}
\theta_t + \theta \cdot \theta_x + (-\lap)^s \theta = 0.
\end{equation}

It is well known that solutions $\theta$ of the subcritical ($s>1/2$) and critical ($s=1/2$) Burgers equation are smooth \cite{kiselev-blow}, \cite{DongDuLi}, \cite{chan2008regularity}. 

There are parallel results for the quasi-geostrophic equation. In the
subcritical case, the solutions are smooth
\cite{MR1709781}. In the critical case the solutions are also smooth,
which was proved independently by Kiselev, Nazarov and Volberg \cite{kiselev2007global} and Caffarelli and Vasseur  \cite{caffarelli2006drift} using different methods. The proof by Kiselev, Nazarov and Volberg
is based on their previous work on the Burgers equation and consists of showing that certain modulus of
continuity (that is essentially Lipschitz for nearby points) is preserved
by the flow. The proof  by Caffarelli and Vasseur is more involved and
consists in proving a H\"older continuity result using classical ideas of
De Giorgi.

The two different methods were also used in the context of the critical Burgers equation.
The method of modulus of continuity was used in \cite{kiselev-blow} to show smoothness of solutions in the periodic setting.  %It was also used in \cite{MiaoWu} to prove global wellposedness in critical Besov spaces $\dot B_{p,1}^{1/p}(\mathbb{R})$.  
On the other hand, the parabolic De Giorgi method developed in  \cite{caffarelli2006drift} was used in \cite{chan2008regularity} to show smoothness of solutions in the non-periodic setting.

For the case of the supercritical quasi-geostrophic equation, it was shown
that the solutions are smooth for large time if $s = 1/2 - \eps$ for a
small $\eps$ \cite{silvestre2008eventual} extending the methods of
Caffarelli and Vasseur. More precisely the idea is to use the \emph{extra
room} in the improvement of oscillation lemma to compensate for the bad
scaling.

In this article, we prove that the solutions of a slightly supercritical
fractional Burger's equation become regular for large time. It is a
similar result to the one shown in \cite{silvestre2008eventual} for the
quasi-geostrophic equation.

It is important to point out that in \cite{kiselev-blow},\cite{ADV},\cite{DongDuLi} it was shown that  
singularities indeed occur for any $s < 1/2$. What we show here is that
they disappear after a certain amount of time. Even though singularities
may (and sometimes do) appear during an interval of time $[0,T]$, for
$t>T$ they do not occur any more. The amount of time $T$ that we need to
wait depends on the initial data and the value of $s$. For any given
initial data, $T \to 0$ as $s \to 1/2$.
The essential idea of the proof is to combine the ideas from
\cite{chan2008regularity} and \cite{silvestre2008eventual}. On the other
hand, we can present a completely self contained proof which has been
simplified considerably.

The idea in the proofs in this paper is still to make the improvement of
oscillation in parabolic cylinders compete with the deterioration of the
equation due to scaling. The improvement of oscillation lemma is the lemma
which allows us to show H\"older continuity when we iterate it at
different scales (as in the classical methods of De Giorgi). We present a
simple and completely self contained proof of this crucial lemma in this
paper (section \ref{s:oscillation}). An alternative approach could be to
redo the proof in \cite{chan2008regularity} adapted to general powers of
the Laplacian using the extension in \cite{caffarelli2007extension}.

We find a few advantages in the choice of presenting this new proof of the
oscillation lemma in this article. One is that it makes the paper self
contained. It also provides a proof that does not use the extension
argument and thus it could be generalized to other integral
operators instead of the fractional Laplacian. The new proof is
essentially a parabolic adaptation of the ideas in
\cite{silvestre2006holder}. This proof uses strongly that the equation is
non-local. This idea is also used in \cite{hj-new} to obtain a H\"older estimate for critical advection diffusion equations for bounded flows that are not necessarily divergence free.

We now state the main result.
\begin{thm}\label{mainthm}
There exists a universal constant $\alpha\in (0,\frac 12)$ such that if $\theta$ is a solution of \eqref{e:burgers} in $\R \times
[0,+\infty]$ with $\frac{1-\alpha}{2}<s\leq \frac 12$ and initial data $\theta_{0}\in L^{2}$, then there exists $T^{\ast}>0$ such that when $t>T^{\ast}$, $\theta(t)$ is $C^{\alpha}$ ($T^*$ depending only on $\norm{\theta_0}_{L^2}$).%, and therefore smooth.
\end{thm}
\begin{remark}
We note that we believe this could be extended to data in any $L^{p}, 1\leq p < \infty$, but for simplicity we do not pursue this here.
\end{remark}
\medskip
\noindent Notation: \\
$Q_r = [-r,r] \times [-r^{2s},0]$.\\
$\osc_{Q_r} \theta = \sup_{Q_r} \theta - \inf_{Q_r} \theta$.
\section{Preliminaries}

\subsection{The notion of a solution and vanishing viscosity approximation}
By a \emph{solution} of \eqref{e:burgers} we mean a weak solution (a solution in the sense of distributions) that can be obtained through the vanishing viscosity method.  In other words it is a limit as $\eps_1 \rightarrow 0$ of solutions satisfying
\begin{equation} \label{e:burger-approx}
\begin{aligned}
\theta_t + \theta \cdot \theta_x + (-\lap)^s \theta - \eps_1 \lap \theta &= 0,\\
\theta(\cdot,0)&=\theta_0 \in L^2(\R),
\end{aligned}
\end{equation}
where $\theta_0$ is an initial data for \eqref{e:burgers}.

For every $\eps_1>0$ and $\theta_0 \in L^2$, the equation \eqref{e:burger-approx} has a solution $\theta$ which is $C^\infty$ for all $t>0$. We list the properties of such solution in the next elementary lemma.

\begin{lemma}\label{l:properties-approx}
For every $\eps_1>0$ and $\theta_0 \in L^2$, the equation \eqref{e:burger-approx} is well posed and its solution $\theta$ satisfies
\begin{enumerate}
\item $\theta(\cdot, t) \in C^{\infty}$ for every $t>0$.
\item Energy equality: \[\norm{\theta(\cdot, t)}_{L^2(\R)}^2 + \int_0^t \norm{\theta(\cdot, t)}_{\dot H^s(\R)}^2 + \eps_1 \norm{\theta(\cdot, t)}_{\dot H^1(\R)}^2 \dd t = \norm{\theta_0}_{L^2(\R)}^2.\]
where $\dot H^s$ stands for the homogeneous Sobolev space.
\item For every $t>0$, $\theta(x,t) \to 0$ as $x \to \pm \infty$.
\end{enumerate}
\end{lemma}

\begin{proof}
We consider the operator that maps $\theta$ to the solution of
\[ \tilde \theta_t + (-\lap)^s \tilde \theta - \eps_1 \lap \tilde \theta =- \theta \ \theta_x. \]
Then we see that the map $A : \theta \mapsto \tilde \theta$ is a contraction in the norm
\[ |||\theta||| = \sup_{[0,T] } \norm{\theta(\cdot,t)}_{L^2} + t^{1/2} \norm{\partial_x \theta(\cdot,t)}_{L^2} \]
To see that we note
\[
|||e^{-t((-\lap)^{s}-\eps_{1}\lap)}\theta_{0}|||\leq C\norm{\theta_{0}}_{L^{2}}.
\]
(This is an elementary computation using Fourier transform). Given $\theta_1$ and $\theta_2$ such that $||| \theta_i ||| \leq R$ for $i = 1,2$, we estimate $|||A\theta_1 - A \theta_2|||$ using Duhamel formula. On one hand we have
\begin{align*}
|| A \theta_1(\cdot,t) - A \theta_2(\cdot,t) ||_{L^2} &\leq C \int_0^t \norm{\theta_1(\cdot,r) \partial_x \theta_1(\cdot,r) - \theta_2(\cdot,r) \partial_x \theta_2(\cdot,r)}_{L^2} \dd r \\
&\leq C \int_0^t \norm{\theta_1-\theta_2}_{L^\infty} \norm{\partial_x \theta_1}_{L^2} + \norm{\theta_2}_{L^\infty} \norm{\partial_x \theta_1 - \partial_x \theta_2}_{L^2} \dd r \\
\intertext{Using the interpolation inequality: $||f||_{L^\infty} \leq ||f||_{L^2}^{1/2} ||f'||_{L^2}^{1/2}$,}
&\leq C R \ |||\theta_1 - \theta_2||| \int_0^t (t-r)^{-1/4} \dd r \leq C R |||\theta_1 - \theta_2||| t^{3/4}.
\end{align*}

On the other hand, we also estimate

\begin{align*}
t^{1/2} || \partial_x A\theta_1(\cdot,t) - \partial_x A \theta_2(\cdot,t) ||_{L^2} &\leq C t^{1/2}  \int_0^t (t-r)^{-{1/2}} \norm{\theta_1(\cdot,r) \partial_x \theta_1(\cdot,r) - \theta_2(\cdot,r) \partial_x \theta_2(\cdot,r)}_{L^2} \dd r \\
&\leq C R t^{1/2} \ |||\theta_1 - \theta_2||| \int_0^t (t-r)^{-3/4} \dd r \leq C R |||\theta_1 - \theta_2||| t^{3/4}
\end{align*}

Thus, if we choose $T$ small enough (depending on $R$), $A$ will be a contraction in the ball of radius $R$ with respect to the norm $||| \cdot |||$.

Therefore, the equation \eqref{e:burger-approx} has a unique solution locally in time for which the norm $|||\cdot|||$ is bounded. A standard bootstrap argument proves that moreover $|||\partial_x^k \theta|||_{L^2} \leq C t^{-k/2}$ for all $k \geq 0$. This proves 1. and 3. for short time.   

The energy equality 2. follows immediately by multiplying equation \eqref{e:burger-approx} by $\theta$ and integrating by parts. Since the $L^2$ norm of the solution is non increasing, the solution can be continued forever, thus 1. and 3. hold for all time.
\end{proof}
 
If we let $\eps_1 \to 0$, the energy estimate allows us to obtain a subsequence of solutions of the approximated problem that converges weakly in $L^\infty(L^2) \cap L^2(\dot H^s)$ to a weak solution for which the energy inequality holds. In a later section, we will also prove a bound of the $L^\infty$ norm of $\theta(\cdot, t)$ for $t>0$, that is also independent of $\eps_1$, thus we can also find a subsequence that converges weak-$\ast$ in $L^{\infty}((t,+\infty) \times \R)$ for every $t>0$. 

\subsection{A word about scaling}
There is a one-parameter group of scalings that keeps the equation
invariant. It is given by $\theta_r = r^{2s-1}\theta(rx, r^{2s}t )$. If
$\theta$ solves \eqref{e:burgers}, then so does $\theta_r$. In the
critical case $s=1/2$, the scaling of the equation keeps the $L^\infty$
norm fixed. This case is critical because the scaling coincides with the a
priori estimate given by the maximum principle.

We can consider a one parameter scaling that preserves H\"older spaces.
The function $\theta_r = r^{-\alpha}\theta(rx, r^{2s}t)$ has the same
$C^\alpha$ semi-norm as $\theta$. If we want to prove that $\theta \in
C^\alpha$, we will have to deal with this type of scaling, but in this
case the equation is not conserved. Instead, if $\theta$ satisfies
\eqref{e:burgers}, $\theta_r$ satisfies
\[ \partial_t \theta_r + r^{2s - 1 + \alpha} \theta_r \cdot \partial_x
\theta_r + (-\lap)^s \theta_r = 0. \]

We have an extra factor in front of the nonlinear term. Note that if
$\alpha > 1-2s$ (only slightly supercritical) and $r < 1$ (zoom in), this
factor is smaller than one.

In the case of the equation with the extra term $\eps_1 \lap \theta$, the viscosity will have a larger effect in smaller scales. Indeed, if $\theta$ satisfies
\eqref{e:burger-approx}, $\theta_r$ satisfies
\[ \partial_t \theta_r + r^{2s - 1 + \alpha} \theta_r \cdot \partial_x
\theta_r + (-\lap)^s \theta_r + r^{2s-2} \eps_1 \lap \theta_r= 0. \]

\section{$L^\infty$ Decay}
First, as an immediate consequence of the energy equality in Lemma \ref{l:properties-approx} we have the following lemma.  
\begin{lemma}\label{l:l2} If $\theta$ is a
solution of \eqref{e:burgers}, then
\[
\norm{\theta(t)}_{L^2(\R)}\leq \norm{\theta_0}_{L^2{(\R)}}.
\]
\end{lemma}
Nonincreasing properties of $L^p$ norms as above for general $1<p\leq \infty$ for the quasi-geostrophic equations were showed in  \cite{Resnickthesis},\cite{CordobaDouble}.   Now we have a theorem about the decay of the $L^\infty$ norm.  See also \cite{kiselev-blow},\cite{caffarelli2006drift},\cite{chan2008regularity},\cite{silvestre2008eventual}.
\begin{thm}\label{thmdecay}
If $\theta$ is a solution of \eqref{e:burgers}, then
\begin{align}
\sup_{x\in \R}\abs{\theta(x,t)}\leq C(s)
t^{-\frac{1}{4s}}\norm{\theta_0}_{L^2(\R)},
\end{align}
where $C(s)=\frac{2s}{C_s^{1/4s}}\sqrt{\tfrac{2}{1+4s}}$, and $C_s$ is the constant appearing the integral formulation of the fractional Laplacian below.
\end{thm}
\begin{proof}
Let $T>0$ and suppose $\theta$ is a solution of \eqref{e:burger-approx}. 
Define
\[
F(x,t)=t^\frac{1}{p}\theta(x,t),
\]
for some $p$ to be chosen later.  By Lemma \ref{l:properties-approx}  there must exist a point $(x_0,t_0)$ such that
\begin{align*}
\sup_{\R\times[0,T]}F(x,t)=F(x_0,t_0)<\infty.
\end{align*}
Observe that $F$ satisfies the following equation
\be\label{Feq}
F_t-\epsilon \lap F+(-\lap)^sF=\frac{1}{pt}F-\frac{1}{t^\frac1p}F\cdot F_x.
\ee
At $(x_0,t_0)$ we have
\begin{align*}
 F_t\geq 0,\quad F_x=0,\quad-\lap F\geq 0.
\end{align*}
Then by \eqref{Feq}
\be\label{Feq1}
 (-\lap)^s F(x_0,t_0)\leq \frac{1}{pt_0}F(x_0,t_0).
\ee
Using $F(x_0,t_0)-F(y,t_0)\geq 0$ for all $y\in \R$, we compute a lower
bound for $(-\lap)^sF(x_0,t_0)$ as follows
\begin{align}
(-\lap)^sF(x_0,t_0)&=C_s \int_\R
\frac{F(x_0,t_0)-F(y,t_0)}{\abs{x_0-y}^{1+2s}}dy\nonumber\\
                   &\geq C_s \int_{\abs{x_0-y}>R}
\frac{F(x_0,t_0)-F(y,t_0)}{\abs{x_0-y}^{1+2s}}dy,\quad\mbox{for any
$R>0$}\nonumber\\
                   &= \frac{C_s}{sR^{2s}}F(x_0,t_0)- C_s\int_{\abs{x_0-y}>R}
\frac{F(y,t_0)}{\abs{x_0-y}^{1+2s}}dy\label{Feq2}.
\end{align}
Next by Cauchy Schwarz
\begin{align}
\int_{\abs{x_0-y}>R} \frac{F(y,t_0)}{\abs{x_0-y}^{1+2s}}dy&\leq
\frac{\tilde {C_s}}{R^{1/2+2s}}\norm{F(t_0)}_{L^2(\R)}\nonumber\\
&=\frac{\tilde {C_s}t_0^\frac{1}{p}}{R^{1/2+2s}}\norm{\theta(t_0)}_{L^2(\R)}\leq\frac{\tilde {C_s}t_0^\frac{1}{p}}{R^{1/2+2s}}\norm{\theta_0}_{L^2(\R)},\label{Feq3}
\end{align}
where the last inequality follows from Lemma \ref{l:l2} and $\tilde {C_s}=(\tfrac{2}{1+4s})^\frac 12.$  Combine
\eqref{Feq1}-\eqref{Feq3} to obtain
\[
\frac{1}{pt_0}F(x_0,t_0)\geq
C_s(\frac{1}{sR^{2s}}F(x_0,t_0)-\frac{\tilde{C_s}t_0^\frac{1}{p}
}{R^{1/2+2s}}\norm{\theta_0}_{L^2(\R)}),
\]
or equivalently
\[
(\frac{C_s}{sR^{2s}}-\frac{1}{pt_0})F(x_0,t_0)\leq
\frac{\tilde{C_s}C_st_0^\frac1p}{R^{1/2+2s}}\norm{\theta_0}_{L^2(\R)}.
\]
Let $p=4s$, and choose $R$ so that $\frac{C_s}{R^{2s}}=\frac{1}{2t_0}$.
Rearranging we have
\[
F(x_0,t_0)\leq C(s)\norm{\theta_0}_{L^2(\R)},
\]
with $C(s)$ as in the statement of the theorem.
Finally, from the definition of $F$
\[
\sup_{\R\times[0,T]}t^\frac{1}{4s}\theta(x,t)\leq C(s)\norm{\theta_0}_{L^2(\R)},
\]
or
\[
\sup_{\R\times[0,T]}\theta(x,t)\leq t^{-\frac{1}{4s}}C(s)\norm{\theta_0}_{L^2(\R)},
\]
and since the estimate is independent of $\epsilon_1$ and $T$ is
arbitrary, the theorem follows (note this gives an upper bound for
$\theta$.  To obtain a lower bound we can redo the proof with $F$ defined
by $-t^\frac{1}{p}\theta(x,t)$.).
\end{proof}

\begin{remark}
Note that an estimate like \eqref{Feq3} could be obtained using any $L^p$ norm instead of $L^2$. We chose to use $L^2$ because it is the norm that is easiest to show that it stays bounded (using the energy inequality). 
\end{remark}

\section{The oscillation lemma}
\label{s:oscillation}

\begin{lemma} \label{l:measureToPointEstimate}
Let $M_{0} \geqslant 2$ and $s \in [\frac{1}{4} , \frac{1}{2}]$.  Assume $\theta\leq 1$ in $\mathbb{R} \times [-\frac{2}{M_{0}} , 0]$ and $\theta$ is a subsolution of
\[
\theta_{t} + M \theta \cdot \theta_{x} + (-\triangle )^{s} \theta - \eps_{1} \triangle \theta \leq \eps_{0}  ,
\]
in the set $[-5, 5 ] \times [-\frac{2}{M_{0}} , 0] $ where $|M| \leq M_{0}$ and $0 < \eps_{1} \leq 10^{3/2}$. Assume also that 
\[
|\{\theta \leq 0\}\cap ( [-1,1] \times [-\frac{2}{M_{0}} , -\frac{1}{M_{0}}]) | \geq \mu.
\]
Then, if $\eps_0$ is small enough (depending only on $\mu$ and $M_0$) there is a $\lambda>0$ (depending only on $\mu$ and $M_0$) such that $\theta \leq 1-\lambda$ in $[-1,1] \times [-\frac{1}{M_{0}} , 0]$.
\end{lemma}

We will apply the lemma above only to the case when $M$ is constant in
$Q_1$. This is not necessary to prove the lemma as it will be apparent in
the proof. We are not aware of any possible application of the lemma with
variable $M$ (even discontinuous).

\begin{proof}
Let $m: [-\frac{2}{M_{0}} , 0] \to \R$ be the solution of the following ODE:
\begin{equation} \label{e:ODEform}
\begin{aligned}
m(- \frac{2}{M_{0}} ) &= 0, \\
m'(t) &= c_0 | \{x \in [-1,1]: \theta(x,t) \leq 0\}| - C_1 m(t).
\end{aligned}
\end{equation}

The above ODE can be solved explicitly and $m(t)$ has the formula
\[ m(t) = \int_{- \frac{2}{M_{0}}   }^t c_0 | \{x : \theta(x,s) \leq 0\} \cap B_1 | e^{-C_1
(t-s)} \dd s. \]

We will show that if $c_0$ is small and $C_1$ is large, then
$\theta \leq 1 - m(t) + \eps_0$ in $[-1,1] \times [- \frac{1}{M_{0}}   ,0]$. This naturally
implies the result of the lemma since for $t \in [- \frac{1}{M_{0}}  ,0]$,
\[ m(t) \geq c_0 e^{- \frac{2C_{1}}{M_{0}}  } |\{ \theta \leq 0 \} \cap [-1,1] \times
[-\frac{2}{M_{0}}  , - \frac{1}{M_{0}}    ] | \geq  c_0 e^{- \frac{2C_{1}}{M_{0}}  }\mu.\]
So we can set $\lambda = c_0 e^{- \frac{2C_{1}}{M_{0}}    }\mu/2$ for $\eps_0$ small. 

Let $\beta : \R \to \R$ be a fixed smooth nonincreasing function such that $\beta(x)=1$ if $x \leq 1$ and $\beta(x)=0$ if $x \geq 2$. Moreover, we can take $\beta$ with only one inflection point between $0$ and $2$, so that if $\beta \leq \beta_0$ then $\beta'' \geq 0$.

Let $b(x,t) = \beta(|x|+ M_{0} t) = \beta (|x| - M_{0} |t|)$. As a function of $x$, $b(x,t)$ looks like a bump function for every fixed $t$. By construction $b_{xx} \geq 0$ if $b \leq \beta_0$. Moreover, at those points where $b = 0$ (precisely where $|x| \geq 2-M_{0}t = 2 + M_{0} |t|)$, $(-\lap)^s b < 0$. Since $b$ is smooth, $(-\lap)^s b$ is continuous and it remains negative for $b$ small enough. Thus, there is some constant $\beta_1$ such that $b_{xx} \geq 0$ and $(-\lap)^s b \leq 0$ if $b \leq \beta_1$.

Assume that $\theta(x,t) > 1 - m(t) + \eps_0 (1+t)$ for some point $(x,t)
\in [-1,1] \times [-\frac{1}{M_{0}},0]$. We will arrive to a contradiction by looking at the maximum
of the function
\[ w(x,t) = \theta(x,t) + m(t) b(x,t) - \eps_0 (1+t). \]
We are assuming that there is one point in $[-1,1] \times [-\frac{1}{M_{0}},0]$ where $w(x,t) > 1$. Let $(x_0,t_0)$ be the point that realizes the maximum of $w$:
\[ w(x_0,t_0) = \max_{\R \times [-\frac{2}{M_{0}},0]} w(x,t).\]
(Note $(x_0,t_0)$ exists by the definition of $w$ and Lemma \ref{l:properties-approx}.)
Since $w(x_0,t_0) > 1$, by using the fact that $\theta (x_0 , t_0) \leq 1$, we deduce $m(t_0) b(x_0, t_0) > \eps_0 (1+t_0) >0$, which further implies $m(t_0) >0$ (this tells us that $t_{0} > -\frac{2}{M_{0}}$) and $b(x_0 , t_0) > 0$, so $|x_0| < 2 + M_{0} |t_{0}| \leq 4 $.

Since the function $w$ realizes a maximum at $(x_0,t_0)$, we have the
following elementary inequalities:
\begin{align*}
w(x_0,t_0) &> 1\\
w_t(x_0,t_0) &\geq 0  \\
w_x(x_0,t_0) &= 0 \\
\lap w(x_0,t_0) &\leq 0 \\
(-\lap)^s w(x_0,t_0) &\geq 0
\end{align*}

The last inequality can be turned into a more useful estimate by recalling
the integral formula of $(-\lap)^s w$ and looking at the set of points
where $\theta \leq 0$.
\[
\begin{aligned}
(-\lap)^s w(x_0,t_0) &= C_s \int_{\R} \frac{ w(x_0,t_0) - w(y,t_0)
}{|x_0-y|^{1+2s}} \dd y \qquad \text{(Note the integrand is nonnegative)}
\\
&\geq C_s \int_{\{y \in [-1,1] : \theta(y,t_0) \leq 0\}} (w(x_0,t_0) -
w(y,t_0)) 5^{-1-2s} \dd y\\
&\geq C_s (1-m(t_0)) 5^{-1-2s} |\{y \in [-1,1] : \theta(y,t_0) \leq 0\}| \\
&\geq \frac{C_{s}}{25} (1-m(t_0))|\{y \in [-1,1] : \theta(y,t_0) \leq 0\}| ,
\end{aligned}
\]
where the last inequality is valid since $5^{1+2s} \leq 25$ for $\frac{1}{4} \leq s \leq \frac{1}{2}$.
We choose the constant $c_0$ in order to make sure that $m(t)$ stays below
$1/4$ (simply by choosing $c_0 < 1/8$), and we choose $c_0 \leq  \frac{3}{4}\frac{C_{s}}{25} $, so
that
\begin{equation} \label{e:bound-fraclap}
(-\lap)^s w(x_0,t_0) \geq c_0 |\{y \in [-1,1] : \theta(y,t) \leq 0\}|.
\end{equation}

Note that the constant $C_s$ in the integral form of the fractional
Laplacian stays bounded and away from zero as long as $s$ stays away from
$0$ and $1$. We can consider $C_s$ bounded above and below independently
of $s$ as long as $s$ stays in a range away from $0$ and $1$, like for
example $s \in [1/4,1/2]$.

Now we recall that $w = \theta + mb - \eps_0 (1+t)$ and we rewrite the inequalities in terms of $\theta$. %Recall that $b$ is a fixed smooth function with compact support, so $b_x$, $\lap b$ and $(-\lap)^s b$ are bounded.
\begin{align*}
1 \geq \theta(x_0,t_0) &\geq 1 - m(t_0) b(x_0,t_0) \geq 3/4 \\
\theta_t (x_0,t_0) &\geq -m'(t_0) b(t_0,x_0) + m(t_0)M_{0} |b_x(x_0,t_0)| + \eps_0\\
\theta_x (x_0,t_0) &= -m(t_0) b_x(x_0 ,t_0) \\
\lap \theta(x_0,t_0) &\leq -m(t_0) \lap b(x_0,t_0) \\
(-\lap)^s \theta(x_0,t_0) &\geq -m(t_0) (-\lap)^s b(x_0,t_0) + c_0 |\{y \in
[-1,1] : \theta(y,t_0) \leq 0\}|
\end{align*}

We consider two cases and obtain a contradiction in both. Either $b(x_0,t_0) > \beta_1$ or $b(x_0,t_0) \leq \beta_1$.

Let us start with the latter. If $b(x_0,t_0) \leq \beta_1$, then $\lap b(x_0,t_0) \geq 0$ and 
$(-\lap)^s b(x_0,t_0) \leq 0$, then
\begin{align*}
\lap \theta(x_0,t_0) &\leq -m(t_0) \lap b(x_0,t_0) \leq 0 \\
(-\lap)^s \theta(x_0,t_0) &\geq c_0 |\{y \in [-1,1] : \theta(y,t_0) \leq 0\}|
\end{align*}

Therefore
\[ \eps_0 \geq \theta_t + M \theta \theta_x + (-\lap)^s \theta - \eps_1
\lap \theta \geq \eps_0 - m'(t_0) b(x_0) + c_0 |\{y \in [-1,1]
: \theta(y,t_0) \leq 0\}| ,
\]
where in the last inequality, we have implicitly use the fact that 
 
\[m(t_0)\big(M_0|b_x(x_0 ,t_0)| - M\theta (x_0 ,t_0 ) b_x(x_0 ,t_0)\big) \geq 0,\] 
since $1\geq\theta (x_0 ,t_0)\geq\frac 34$ and $|M|\leq M_0$.

So we obtain
\begin{equation*}
 - m'(t_0) b(x_0) + c_0 |\{y \in [-1,1] : \theta(y,t_0) \leq 0\}| \leq 0,
\end{equation*}
but this is a contradiction with \eqref{e:ODEform} for any $C_1 \geq 0$.

Let us now analyze the case $b(x_0,t_0) > \beta_1$. Since $b$ is a smooth, compactly supported function, there is some constant $C$ (depending on $M_{0}$), such that $|\lap b| \leq C$ and $|(-\lap)^s b| \leq C$.
Then we have the bounds
\begin{align*}
\lap \theta(x_0,t_0) &\leq -m(t_0) \lap b(x_0,t_0) \leq C m(t_0) \\
(-\lap)^s \theta(x_0,t_0) &\geq c_0 |\{y \in [-1,1] : \theta(y,t_0) \leq 0\}| - C m(t_0)
\end{align*}
Therefore
\[ \eps_0 \geq \theta_t + M \theta \theta_x + (-\lap)^s \theta - \eps_1
\lap \theta \geq \eps_0 - m'(t_0) b(x_0 ,t_0 ) - C m(t_0) + c_0 |\{y \in [-1,1]
: \theta(y,t_0) \leq 0\}|
\]
and we have
\begin{equation*} %\label{e:f0}
 - m'(t_0) b(x_0,t_0) - C m(t_0) + c_0 |\{y \in [-1,1] : \theta(y,t_0) \leq 0\}|
\leq 0.
\end{equation*}

We replace the value of $m'(t_0)$ in the above inequality using
\eqref{e:ODEform} and obtain
\[  (C_1 b(x_0,t_0)- C) m(t_0) + c_0 (1- b(x_0,t_0)) |\{y \in [-1,1] : \theta(y,t)
\leq 0\}| \leq 0.  \]
Recalling that $b(x_0,t_0) \geq \beta_1$, we arrive at a contradiction if $C_1$ is
chosen large enough.
\end{proof}

\begin{lemma} \label{l:oscillation-approx}
Let $s \in [\frac{1}{4} , \frac{1}{2}]$, and let $\theta$ be a solution of
\be\label{e:osc2}
 \theta_t + M \theta \cdot \theta_x + (-\lap)^s \theta - \eps_1 \lap
\theta \leq 0,
\ee
where $|M|\leq 1$ and $\eps_1 \leq 1$.  Assume that $|\theta| \leq 1$ in $Q_1$ and
$|\theta(x)| \leq |500x|^{2\alpha}$ for $|x|>1$.  Then if $\alpha$ is small
enough, there is a $\lambda>0$ (which does not depend on $\eps_1$) such
that $\osc_{Q_{1/400}} \theta \leq 2-\lambda$.
\end{lemma}

There is no deep reason for the choice of the number $500$ in the above lemma. But the smaller the cube is, say $Q_{\frac{1}{400}}$, on which the improved oscillation occurs, we need a number, say $500$, which is greater than $400$ in order to make inequality \eqref{c2} hold.  In principle, $500$ can be replaced by any number greater than $400$.

\begin{proof}
We want to apply Lemma \ref{l:measureToPointEstimate} to $\theta$.  We check if we have the required hypothesis.   We set $M_0=2\cdot10^{1/2}$.  (The reason for this choice will become clear shortly.)  Next, $\theta$ will be either nonnegative or nonpositive in half of
the points in  $[-10,10] \times
[-\frac{2}{M_{0}} , -\frac{1}{M_{0}}] $ (in measure).   Let us assume
$| \{ (x,t) \in [-1,1] \times [-\frac{2}{M_{0}} , -\frac{1}{M_{0}}] : \theta(x,t) \leq 0 \} | \geq \mu =\frac{1}{M_0}$.
(Otherwise, we would continue the proof with
$-\theta$ instead of $\theta$ and $-M$.)  
Next, the hypothesis that we are missing is that $\theta $ may be larger than
$1$ outside $Q_1$.   Thus we define
\[ \overline \theta = \min(\theta,1) .\]  
We show below $\overline \theta$ satisfies
 \be\label{e:osc21}
\overline \theta_{t} + M  \overline \theta \cdot \overline \theta_{x} +
(-\triangle)^{s} \overline \theta - \eps_{1} \triangle \overline \theta \leq
\eps_{0}.
\ee
over $Q_{1/2}$ for $\epsilon_{0}$ small enough.    Since $\theta$ satisfies \eqref{e:osc2} and  $\overline \theta=\theta$ on $Q_1$ we must only check the difference of $(-\lap)^s \theta$ and
$(-\lap)^s \overline \theta$ since this is the only nonlocal term in the
equation. Let $|x| \leq 1/2$ (note below that we cannot take $x\in Q_1$)
\[ \begin{aligned}
(-\lap)^s \overline \theta(x,t) - (-\lap)^s \theta(x,t) &= C_s \int_{\R} 
\frac{ \overline \theta(x,t) - \theta(x,t) - \overline \theta(y,t) +
\theta(y,t) }{|x_0-y|^{1+2s}} \dd y \\
&=  C_s \int_{\{y : \theta(y,t) > 1\}}  \frac{ \theta(y,t) - 1
}{|x_0-y|^{1+2s}} \dd y \\
& \leq C \int_{\{|y| > 1\}} \frac{|500y|^{2\alpha } - 1 }{|y|^{\frac{3}{2}}}
\dd y =: \omega(\alpha),
\end{aligned}\]
where, in the last inequality, we have used the assumption that $ \frac{1}{4} \leq s \leq \frac{1}{2}$.   
Notice $\omega(\alpha) \to 0$ as $\alpha \to 0$. So we can choose
$\alpha>0$ such that $\omega(\alpha) < \eps_0$.   
Hence $\overline \theta$ satisfies \eqref{e:osc21} over $Q_{1/2}$ as claimed.

However, in order to apply Lemma \ref{l:measureToPointEstimate}, we need to rescale so that we can have that the inequality holds on $[-5,5]\times [-\frac 2 M_0, 0]$.  Since we also need to preserve the condition $\overline \theta   \leq 1$ after rescaling, we choose
to work with  the function $ \overline \theta^{*} (x,t) = \overline \theta (\frac{1}{10} x ,
\frac{1}{10^{2s}} t) $.  Observe that $\overline \theta^{*}$ satisfies the following differential inequality over $Q_5$.

\begin{equation}
\overline \theta^{*}_{t} + 10^{1-2s} M \overline \theta^{*} \cdot \overline
\theta^{*}_{x} + (-\triangle )^{s} \overline \theta^{*}
- 10^{2-2s} \eps_{1} \triangle \overline \theta^{*} \leq \frac{\eps_{0}}{10^{2s}}
\leq \eps_{0}.   
\end{equation}

Observe that with $M_0=2 \cdot 10^{\frac{1}{2}}$, $[-5,5]\times [-\frac 2 M_0, 0] \subset Q_5$, and $10^{1-2s} |M | \leq M_{0}$.  Also $10^{2-2s} \eps_{1} \leq 10^{3/2}$,  and since by construction
$\overline \theta^{*}\leq 1 \in \R \times [-\frac 2 M_0, 0]$,  we now finally can apply Lemma \ref{l:measureToPointEstimate} and obtain that $ \overline \theta^{*} \leq
1 - \lambda $ over $[-1,1]\times [-\frac{1}{M_{0}} , 0]$, where $\lambda$
depends only on $M_{0} = 2\cdot 10^{\frac{1}{2}}$.  However, since we would like to have an improved oscillation on a parabolic cube, we note that $Q_{1/40}=[-\frac{1}{40} ,
\frac{1}{40}] \times [-\frac{1}{40^{2s}} , 0] \subset [-1,1] \times
[-\frac{1}{M_{0}} , 0]$, for  $\frac{1}{4} \leq s \leq \frac{1}{2}$.  So we
have $\overline \theta^{*} \leq 1- \lambda  $ over $Q_{1/40}$ .   Hence by rescaling
$\theta =\overline \theta \leq 1- \lambda$ in $Q_{1/400}$.  This completes the proof.
\end{proof}

%\begin{lemma} \label{l:oscillation-noeps1}Let $\frac{1}{4} < s < \frac{1}{2}$, and $\theta$ be an entropy solution of
%\begin{equation}
%\partial_{t} \theta + M \theta \cdot \partial_{x} \theta + (-\triangle)^{s} \theta \leqslant 0,
%\end{equation}
%where $|M| \leqslant M_{0}$. Assume that $|\theta | \leqslant 1$ in $Q_{1}$ and $|\theta (x)| \leqslant 10^{2\alpha } |x|^{2\alpha } $, for $|x| \geqslant 1$. Fix $\rho \in (0, \frac{1}{4})$. Then, there exists $\alpha_{0} \in (0, \frac{1}{2} )$ (depending only on $M_{0}$ and $\rho$), such that if $0 < \alpha < \alpha_{0}$, we have $\osc_{Q_{\frac{1}{4}}} \theta \leqslant (2- \lambda ) = 2 \rho^{\alpha }$.
%\end{lemma}

%\begin{lemma}
%Let $\theta$ be an entropy solution of
%\[ \theta_t + M \theta \cdot \theta_x + (-\lap)^s \theta \leq 0 \]
%where $|M|\leq M_0$. Assume that $|\theta| \leq 1$ in $Q_1$ and
%$|\theta(x) \leq |10 x|^\alpha$ for $|x|>1$.  Then if $\alpha$ is small
%enough, there %is a $\lambda>0$ such that $\osc_{Q_{1/4}} \theta \leq
%$2-\lambda$.
%\end{lemma}
%\begin{proof}
%Somehow this should be implied by the previous one taking $\eps_1 \to 0$.
%\end{proof}

\section{Proof of the main result}
To simplify the exposition of the proof of theorem \ref{t:main-viscosity}, we first
state and establish the following technical but elementary lemma.
\begin{lemma}\label{l:techlemma}
For any $\rho \in (0 , \frac{1}{400})$, there exists some $\alpha_{1} \in
(0, \frac{1}{2} )$, depending only on $\rho$, such that for any $0 <
\alpha < \alpha_{1}$, the following holds:
\begin{align}
1&<\frac{1}{400 \rho } - \frac{1}{\rho } (1- \rho^{\alpha }),  \label{c1}\\
\rho^{-\alpha } (2 - \rho^{\alpha })& < 500^{2\alpha } \{  \frac{1}{400 \rho} - \frac{1}{\rho } (1- \rho^{\alpha })         \}^{2\alpha },\label{c2}\\
\rho^{-\alpha } ( 500^{2\alpha } + 1 - \rho^{\alpha }     ) & <  500^{2
\alpha } \{  \frac{1}{ \rho } - \frac{1}{\rho } (1- \rho^{\alpha })       
 \}^{2\alpha }.\label{c3}
\end{align}
\end{lemma}
\begin{proof}
\eqref{c1} is immediate by the assumptions on $\rho$.  So is \eqref{c2}
after we observe that it is equivalent to
\[
\rho^{-\frac 12}(2-\rho^\alpha)^\frac{1}{2\alpha} <
500\left(\frac{1}{400\rho}- \frac{1}{\rho}(1-\rho^\alpha)\right).
\]
Since $\lim_{\alpha \rightarrow 0} \rho^{-\frac 12} (2-\rho^\alpha)^\frac{1}{2\alpha} = \frac{1}{\rho} < \frac{500}{400\rho} =\lim_{\alpha \rightarrow 0}  500\left(\frac{1}{400\rho}-   \frac{1}{\rho}(1-\rho^\alpha)\right),$  by continuity, the above inequality holds for sufficiently small $\alpha > 0$.

We rearrange \eqref{c3}, and note that it follows from showing that
\[
f(\alpha)=\rho^{\alpha}(500^{2\alpha}+1-\rho^\alpha)-500^{2\alpha}
\rho^{2\alpha^2},
\]
has a local maximum at $0$. This is indeed true, since
$f(0)=f'(0)=0$, and
\[
f''(0)=\ln \rho (4\ln 500-4-2\ln \rho)<0 ,
\]
for any fixed $\rho \in (0 , \frac{1}{400})$.
\end{proof}
\begin{thm} \label{t:main-viscosity}
Let $\theta$ be a solution of \eqref{e:burger-approx} with $|\theta| \leq 1$ in
$\R \times [-1,0]$. There is a small $\alpha \in (0, \frac{1}{2})$ such that if
$\frac{1- \alpha}{2} < s < 1/2$ then $\theta$ satisfies
\[ |\theta(y,0) - \theta(x,0)| \leq C |x-y|^\alpha \]
for some constant $C$ (independent of $\eps_1$) and for all points such that $|x-y|> c \eps_1^{2-2s}$.
\end{thm}

\begin{proof} 
Fix $\rho \in (0, \frac{1}{400})$.   Let $\alpha_{0}$, and $\alpha_{1}$ be as in Lemma \ref{l:oscillation-approx},
and Lemma \ref{l:techlemma} respectively.  Take $\alpha = \min \{
\frac{\alpha_{0}}{2} , \frac{\alpha_{1}}{2}\}$ ($\alpha$ depends only on $\rho$).    Next let $\lambda$ be as in Lemma \ref{l:oscillation-approx}.  Then  if necessary, we can either make $\lambda$ or $\alpha$ smaller, so that $2-\lambda=2\rho^{\alpha}$.  Finally, set $\frac{ 1 -
\alpha }{2} < s < \frac{1}{2} $.

We define the sequence $\theta_k$ recursively for all nonnegative integers $k$ such that $\rho^{(2-2s)k} \geq \eps_1$. We will do it so that every $\theta_k$ satisfies
\begin{align}
\partial_t \theta_k + M_k \theta_k \partial_x \theta_k + (-\lap)^s
\theta_k - \rho^{(2s-2)k} \eps_1 \lap \theta_{k}&= 0 \ \ \text{in } Q_1 \text{ with } M_k \leq 1, \label{e:h1}\\
|\theta_k(x,t)| &\leq 1 \ \ \text{for } (x,t) \in Q_1, \label{e:h2}\\
|\theta_k(x,t)| &\leq 500^{2\alpha} |x|^{2\alpha} \ \ \text{for } |x| \geq
1 \text{ and } t \in [-1,0], \label{e:h3}
\end{align}
For all $k$, we will have $\theta_k(x,0) = \rho^{-\alpha k} \theta(\rho^k
x,0)$.  So \eqref{e:h2} implies
immediately the result of this theorem.

We have to construct the sequence $\theta_k$. We start with $\theta_0 =
\theta$ and $M_0=1$ which clearly satisfy the assumptions. Now we define
the following ones recursively. Let us assume that we have constructed up
to $\theta_k$ and let us construct $\theta_{k+1}$.

Given the assumptions \eqref{e:h1}, \eqref{e:h2} and \eqref{e:h3}, we can
apply Lemma \ref{l:oscillation-approx} as long as $\eps_1 < \rho^{(2-2s)k}$ and obtain that $\osc_{Q_{1/400}}
\theta_k \leq 2-\lambda=2\rho^\alpha$.  If $\eps_1 \geq \rho^{(2-2s)k}$, we stop the iteration, i.e., we iterate only until the viscosity term becomes large.

Since $\osc_{Q_{1/400}} \theta_k \leq 2-\lambda$, there is a number $d \in
[-\lambda/2,\lambda/2]$ such that
\be\label{m1}
-1+\lambda/2 \leq \theta_k-d \leq 1 - \lambda/2 ,\quad \forall (x,t) \in Q_{1/400}. 
%(-\frac{1}{4} , \frac{1}{4})\times [-\frac{1}{4^{2s}} , 0].
\ee
Now we define $\theta_{k+1}$ as follows,
\[ \theta_{k+1} (x,t) = \rho^{-\alpha} [\theta_k\big( \rho (x + L_t),\rho^{2s}t\big) - d], \]
where $L_t=\rho^{2s-1}M_k d t$.
The function $\theta_{k+1}$ satisfies the equation
\[\partial_t \theta_{k+1} + \rho^{\alpha + 2s - 1} M_k \theta_{k+1}
\partial_x \theta_{k+1} + (-\lap)^s \theta_{k+1} - \rho^{(2s-2)(k+1)} \eps_1 \lap \theta_{k} = 0\]
so we define $M_{k+1} = \rho^{\alpha + 2s - 1} M_k$. Due to the fact that $\alpha + 2s -1 >0$ 
for our choice of $s \in (\frac{1- \alpha }{2} , \frac{1}{2})$, we have $M_{k+1}
\leq M_k$.  Hence, we know that $\theta_{k+1}$ satisfies \eqref{e:h1}.
 
Now, since the graph of $500^{2\alpha} |x|^{2\alpha}$ is symmetric about the y-axis, without loss
of generality, suppose $d<0$, so $L_t>0$. 

To establish \eqref{e:h2} for $\theta_{k+1}$, we first note that by \eqref{m1} we have
\be\label{m2}
-1 + \lambda/2\leq\theta_k\big( \rho (x + L_t),\rho^{2s}t\big) - d \leq 1 - \lambda/2 ,\quad \forall x \in  [-\frac{1}{400\rho}-L_t,\frac{1}{400\rho}-L_t], t\in [0,1]. 
\ee
Next we show that the absolute value of the transport term
$L_t = \rho^{2s-1} M_k d t $ is small enough, so that
$[-1,1]\subset [-\frac{1}{400\rho}-L_t,\frac{1}{400\rho}-L_t]$.  
Indeed, since $M_k d t\leq
\frac{\lambda}{2}=(1-\rho^\alpha)$ we have
\begin{align*}
\frac{1}{400\rho}-\rho^{2s-1} M_k d t&\geq
\frac{1}{400\rho}-\rho^{2s-1}(1-\rho^\alpha)\\
                                                                                                                                         &\geq
\frac{1}{400\rho}-\frac{1}{\rho}(1-\rho^\alpha) >
1,
\end{align*}
which holds by $\eqref{c1}$.  We conclude $[-1,1]\subset [-\frac{1}{400\rho}-L_t,\frac{1}{400\rho}-L_t] $.  Thus by \eqref{m2} for all $(x,t) \in Q_{1}$
\[ |\theta_{k+1}(x,t)| \leq \rho^{-\alpha} |\theta_k\big( \rho (x + L_t),\rho^{2s}t\big)-d| \leq \frac 1 {1-\lambda/2} (1 - \lambda/2 ) = 1, \]
so \eqref{e:h2} holds as needed.

Now we introduce
\[
\psi(x)=\left\{\begin{array}{l} 1 \ \ \quad\quad\quad\quad\quad\mbox
{if}\mbox\quad \abs{x}<1,\\
                                                                                                                                                                                                                        500^{2\alpha}\abs{x}^{2\alpha} \quad\mbox{
if}\mbox\quad \abs{x}\geq 1 .        
\end{array}\right.
\]
By the inductive hypothesis
\[
|\theta_k(x,t)| \leq \psi (x),\quad t \in [-1,0].
\]
Then observe that by definition of $\theta_{k+1}$, in order to establish
\eqref{e:h3} for $\theta_{k+1}$, it is enough to show
\be\label{m3}
\big(\rho^{-\alpha}\psi(\rho(x+L_t))+\rho^{-\alpha} |d|\big) \chi_{\{   |x +
L_{t}| \geqslant \frac{1}{400\rho}         \}  } \leq \psi(x).
\ee
%since $|\theta_{k+1}| \chi_{(-\frac{1}{400\rho} -L_{t} , \frac{1}{400\rho} - L_{t})} \leqslant 1$.
%(Since then, we will have $|\theta_{k+1}| \chi_{\{|x + L_{t}| \geqslant \frac{1}{400\rho}  \}} < \psi$, 
%which together with $|\theta_{k+1}| \chi_{(-\frac{1}{400\rho} -L_{t} , \frac{1}{400\rho} - L_{t})} \leqslant 1$ will %gives $|\theta_{k+1}| \leqslant \psi$.)

First we note that
\[
\big(\rho^{-\alpha}\psi(\rho(x+L_t))+\rho^{-\alpha} |d|\big) \chi_{\{ |x + L_{t}| \geqslant
\frac{1}{400\rho } \}} \leqslant \phi_{1}(x)  + \phi_{2}(x),
\]
where $\phi_{1}(x) = \rho^{-\alpha } (2- \rho^{\alpha } ) \chi_{\{
\frac{1}{400\rho } \leqslant |x + L_{t}| < \frac{1}{\rho} \} }$
and $\phi_{2}(x) =  \{ \rho^{-\alpha} \psi ( \rho ( x +  L_{t} )  )  +
\rho^{-\alpha } (1-\rho^{\alpha }) \} \chi_{ \{ |x + L_{t}| \geqslant
\frac{1}{\rho } \} }$.
So \eqref{m3} will follow if we can show that $\phi_{1} < \psi$ and $\phi_{2} < \psi$.

To show $\phi_{1} < \psi$, we observe that, by \eqref{c2} we have
\[
\phi_{1} (\frac{1}{400\rho } - L_{t} ) =  \rho^{-\alpha } (2- \rho^{\alpha
} ) <   \psi ( \frac{1}{400\rho } - \frac{1}{\rho } (1- \rho^{\alpha }) )  
    \leqslant   \psi (   \frac{1}{400\rho } - \rho^{2s-1} (1- \rho^{\alpha
})     )
\leqslant \psi ( \frac{1}{400\rho } - L_{t}   ).
\]
Since $ \phi_{1} $ is constant over $[\frac{1}{400\rho } - L_{t} , \frac{1}{\rho} - L_{t}]$, and $\psi (x)$ is strictly increasing for
$x \geqslant \frac{1}{400\rho} - L_{t} $, it follows that  $ \phi_{1} (\frac{1}{400\rho } - L_{t} ) <  
\psi ( \frac{1}{400\rho } - L_{t}   ) $ implies $\phi_{1} \chi_{[\frac{1}{400\rho } - L_{t} , \frac{1}{\rho} - L_{t}]  } < \psi $. On the other hand, 
it is quite obvious that we must have $ \phi_{1} \chi_{[  -\frac{1}{ \rho } - L_{t} , -\frac{1}{400\rho} - L_{t}]  } < \psi.$  Hence we deduce that $\phi_{1} < \psi$.

To prove $ \phi_{2} <  \psi  $, we just need to observe that by \eqref{c3}
\[
\phi_{2} ( \frac{1}{\rho } - L_{t}  ) = \rho^{-\alpha} \{ 500^{2\alpha} + 1
- \rho^{\alpha } \} < \psi ( \frac{1}{\rho } -\frac{1}{\rho }
(1-\rho^{\alpha })  )
\leqslant     \psi (   \frac{1}{ \rho } - \rho^{2s-1} (1- \rho^{\alpha }) 
   )
\leqslant \psi (\frac{1}{\rho }  - L_{t} ).
\]

Now, for any point $x \in [ \frac{1}{\rho } -
L_{t} , +\infty  )$ the derivative of $\phi_{2}$ at $x$ is strictly less
than the derivative of $\psi$ at $x$.  Because of this, $ \phi_{2} (
\frac{1}{\rho } - L_{t}  ) < \psi (\frac{1}{\rho }  - L_{t} )   $ at once
implies that $\phi_{2} \chi_{[ \frac{1}{\rho} - L_{t} , +\infty )} < \psi$. 
On the other hand, we also have $\phi_{2} \chi_{(-\infty , -\frac{1}{\rho} - L_{t}    ]} < \psi$. Hence we conclude that $\phi_{2} < \psi$, and this completes the proof. 
\end{proof}

\begin{cor} \label{c:main1}
Let $\theta$ be a solution of \eqref{e:burgers} with $|\theta| \leq 1$ in
$\R \times [-1,1]$. There is a small $\alpha \in (0, \frac{1}{2})$ such that if
$\frac{1- \alpha}{2} < s < 1/2$ then $\theta(\cdot, t) \in C^\alpha$ for all $t \geq 0$.
\end{cor}

\begin{proof}
For every $\eps_1$, we have a solution $\theta^{\eps_1}$ of \eqref{e:burger-approx} for which we can apply Theorem \ref{t:main-viscosity} in any interval of time $[-1+t,t]$. Since neither constant $\alpha$ or $C$ depend on $\eps_1$, then for any $h \in \R$,
\[ \theta^{\eps_1}(x+h,t) - \theta^\eps(x,t) \leq C |h|^\alpha \ \ \text{for all $x \in \R$ and $t \in [0,1]$}\]
for all $\eps_1$ small enough (depending on $|h|$). This estimate passes to the limit as $\eps_1 \to 0$ since $\theta^{\eps_1}(\cdot,0) \to \theta(\cdot,0)$ weak-$\ast$ in $L^\infty$. Moreover, it will hold for all $h$ at the limit, which finishes the proof.
\end{proof}

%\begin{thm} \label{t:main2}
%Let $\theta$ be a solution of \eqref{e:burgers} in $\R \times
%[0,+\infty]$. There is a $T>0$ such that $\theta$ is smooth for $t>T$.
%\end{thm}
Now the proof of the main result follows immediately.
\begin{proof}[Proof of Theorem \ref{mainthm}]
For any initial data $\theta_0 \in L^2$, by Theorem \ref{thmdecay}
$\norm{\theta(-,t)}_{L^\infty(\R)}$ decays. So all we have to do is wait
until it is less than one, and we can apply Corollary \ref{c:main1}.
\end{proof}

\begin{remark}
The only part of the paper where we use that the solution is in $L^2$ is in the proof of the decay of the $L^\infty$ norm (Theorem \ref{thmdecay}). For the rest of the paper, all we use is that the $L^\infty$ norm of $\theta$ will eventually become smaller than one so that we can apply Corollary \ref{c:main1}. Of course there is nothing special about the number one, and a similar estimate can be obtained just by assuming that $\norm{\theta}_{L^\infty} \leq C$. However, the value of $\alpha$ would depend on this $C$.
\end{remark}

\section*{Acknowledgment}

Luis Silvestre was partially supported by NSF grant DMS-0901995 and the Alfred P. Sloan foundation.

\bibliographystyle{plain}
\bibliography{scb}
\end{document}